\documentclass[a4paper,11pt, reqno]{amsart}
\usepackage[DIV=12, oneside]{typearea} 
\usepackage[utf8]{inputenc}
\usepackage[T1]{fontenc}
\usepackage[english]{babel}
\usepackage{esint}
\usepackage{amssymb, amsthm, bbm}
\usepackage{enumerate, mathrsfs, hyphenat, graphicx}
\usepackage{thmtools, thm-restate}
\usepackage{thm-patch, thm-kv, thm-autoref}
\usepackage[colorlinks=true, urlcolor=blue, linkcolor=blue, citecolor=black, pdfpagelabels]{hyperref}
\theoremstyle{plain}
\declaretheorem[title=Theorem, parent=section]{sa}
\declaretheorem[title=Lemma,sibling=sa]{lem}

\declaretheorem[title=Proposition,sibling=sa]{prop}
\newtheorem*{thm*}{Satz}
\theoremstyle{definition}
\declaretheorem[title=Definition,sibling=sa]{defi}

\numberwithin{equation}{section}
\newcommand{\R}{  \mathbb{R}}

\newcommand{\E}{  \mathbb{E}}

\newcommand{\cE}{  \mathcal{E}}
\newcommand{\cF}{  \mathcal{F}}

\newcommand{\cU}{  \mathcal{U}}

\renewcommand{\epsilon}{\varepsilon}

\newcommand{\norm}[1]{\left\| #1 \right\|}
\newcommand{\bet}[1]{\left| #1 \right|}

\renewcommand{\d}{\, \textnormal{d}}

\newcommand{\akon}{\ensuremath{(2-\alpha)}}

\def\E{{\mathcal E}}
\def\F{{\mathcal F}}

\def\Rscript{{\mathcal R}}

\def\real{{\mathbb R}}

\def\Indicator{{\mathbbm{1}}}
\def\integer{{\mathbb Z}}

\def\ep{\varepsilon}
\def\al{\alpha}

\def\Lam{\Lambda}

\def\grad{\nabla}
\def\Div{\textnormal{div}}

\def\d{\textnormal{d}}

\def\intersect{\bigcap}

\newcommand{\abs}[1]{\left| #1 \right|}
\newcommand{\eps}{\varepsilon}

\usepackage[usenames,dvipsnames]{xcolor} 

\addto\extrasenglish{
\renewcommand{\sectionautorefname}{Section}%
}

\begin{document}
\renewcommand{\sectionautorefname}{Section}
\title{Regularity results for nonlocal parabolic equations}
\author{Moritz Kassmann, Russell W. Schwab}
\thanks{Support by the German Science Foundation DFG (SFB 701) is gratefully
acknowledged.}
\begin{abstract}
We survey recent regularity results for parabolic equations involving nonlocal
operators like the fractional Laplacian. We extend the results of
\cite{FeKa13} and obtain regularity estimates for
nonlocal operators with kernels not being absolutely continuous with respect to
 the Lebesgue measure. 
\end{abstract}

\address{Fakultät für Mathematik\\
Universität Bielefeld \\
Postfach 100131 \\
D-33501 Bielefeld \\
Germany}
\email{moritz.kassmann@uni-bielefeld.de}

\address{Department of Mathematics\\
Michigan State University \\
619 Red Cedar Road \\
East Lansing, MI 48824 \\
United States of America}
\email{rschwab@math.msu.edu}

\date{\today}

\maketitle


%
%
%


\section{Introduction}

In this article we discuss the regularity properties-- such as
the weak Harnack inequality and H\"older regularity-- for functions, $u:I\times
\R^{d} \to \R$, which are solutions of 
\begin{align}\label{eq:par_equation}
 \partial_t u(t,x) - Lu(t,x) = f(t,x), \qquad (t,x) \in I \times \Omega,
\end{align}
where $\Omega$ denotes a bounded domain in $\R^d$ and $I$ an open,
bounded interval in $\R$. $L$ is an integro-differential operator of the form

\begin{align}\label{eq:def_L}
 L u (t,x)= \operatorname{p.v.}  \int_{\R^d} \left[ u(t,y)-u(t,x)\right]
a(t,x,y)
\mu(x,\d y) \,.
\end{align}

Previous results (discussed in Section \ref{sec:review}) have focused on the
case in which $\mu(x,dy)$ is absolutely continuous with respect to the Lebesgue
measure, and so the key motivation for this note is to prove that in fact
similar results still hold when $\mu$ is only a measure.  Furthermore, we take
this article as an opportunity to present a small survey on a rapidly expanding
and important family of results surrounding the study of regularity properties
of solutions of (\ref{eq:par_equation}).

We first present the set-up and assumptions for (\ref{eq:def_L}) and our main results, and then in Sections \ref{sec:review} and \ref{sec:examples} we provide a somewhat extensive discussion for (\ref{eq:par_equation}).

In \cite{FeKa13} properties of solutions are studied under the assumption that
the measures $\mu(x,\d y)$ are absolutely continuous with respect to the
Lebesgue-measure in $\R^d$. Here we try to avoid this assumption. We assume that
$a: [0,\infty) \times \R^d \times \R^d \to [1, 2]$ satisfies $a(t,x,y) =
a(t,y,x)$ for all $t,x,y$ and $( \mu(x,\cdot) )_{x \in \R^d}$ is a family of
measures satisfying the following assumptions. On one hand, we assume the
symmetry condition that for
every set $A \times B \subset \R^d \times \R^d$ 
\begin{equation}\label{IntroEq:MuSymmetry}
 \int\limits_A \int\limits_B  \mu(x,\d y) \, dx =  \int\limits_B \int\limits_A 
\mu(x,\d y) \, dx \,. 
\end{equation}

On the other hand, we need to impose conditions on the singular behaviour of
the measures $\mu(x,\d y)$ at the diagonal $\{x=y\}$. We assume that for
some $\alpha \in (0,2)$ and some $\Lambda \geq 1$ and every
$x_0 \in \R^d$, $\rho > 0$ and $v \in
H^{\alpha/2}(B_{\rho}(x_0))$
\allowdisplaybreaks
\begin{gather*}
\rho^{-2} \int_{\bet{x_0-y}\leq \rho} \bet{x_0-y}^2 \mu(x_0,\d y) + \int_{\bet{x_0-y}>\rho} \mu(x_0,\d y) \leq \Lambda \rho^{-\alpha}, \tag{K$_1$} \label{eq:K-1}\\
\begin{split}
 \Lambda^{-1}  \iint\limits_{B\,B} \left[ v(x)-v(y)\right]^2 & \mu(x,\d y)  \d x \leq 
(2-\alpha) \iint\limits_{B\,B}   \frac{\left[ v(x)-v(y)\right]^2}{\bet{x-y}^{d+\alpha}} \d x \d y \\
&\leq \Lambda  \iint\limits_{B\,B}  \left[ v(x)-v(y)\right]^2 \mu(x,\d y)  \d x,\quad \text{where }B=B_{\rho}(x_0).  
\end{split}\tag{K$_2$} \label{eq:K-2}
\end{gather*}

{\bf Remark:} We assume these conditions to hold for all $\rho > 0$
because it simplifies our presentation. It would be sufficient to assume
\eqref{eq:K-1}, \eqref{eq:K-2} for small $\rho$, say for all $\rho \in (0,1]$. 

{\bf Remark:} A standard example satisfying the above assumptions is given by
$\mu^{(\alpha)}(x,\d y) = |x-y|^{-d-\alpha} dy$ for some $\alpha
\in
(0,2)$. In this case the constant $\Lambda$ would depend on $\alpha$ and blow
up for $\alpha \to 2-$. This would affect our results because the constants in
our theorems depend (only) on $d$, a lower bound for $\alpha$
and $\Lambda$. If, instead, one chooses $\mu^{(\alpha)}(x,\d y) =
(2-\alpha) |x-y|^{-d-\alpha} dy$, then
assumptions \eqref{eq:K-1}, \eqref{eq:K-2} are satisfied for one fixed
$\Lambda$ and all for all $\alpha$ satisfying $\alpha_0 \leq \alpha < 2$. In
this sense our results are robust, i.e. the constants stay bounded for $\alpha
\to 2-$. This is a natural assumption considering the fact that it is
$\mu(x,dy)=(2-\al)\abs{x-y}^{-d-\al}\d y$, $a(t,x,y)=1$, which ensures
(\ref{eq:def_L}) to converge to $c(d)(-\Delta)u$ as $\al \to 2-$, where $c(d)$
is a dimensional constant.

Here are our main results: 


\begin{sa}[weak Harnack inequality] \label{thm:harnack} Assume \eqref{eq:K-1}
and \eqref{eq:K-2} hold true for some $\Lambda \geq 1$ and $\alpha \in
(\alpha_0,2)$. There is a constant $C=C(d,\alpha_0,\Lambda)$ such that for every
supersolution $u$ of \eqref{eq:par_equation} on $Q = (-1,1) \times B_2(0)$ which
is nonnegative in $(-1,1) \times \R^d$ the following inequality holds:
\begin{equation*} \label{eq:harnack-speziell}
 \norm{u}_{L^1(U_\ominus)} \leq C \left( \inf_{U_{\oplus}} u + \norm{f}_{L^\infty(Q)} \right) \tag{HI}
\end{equation*}
where $U_\oplus=\left(1-(\tfrac12)^\alpha,1\right) \times B_{1/2}(0)$, $U_\ominus=\left( -1, -1+(\tfrac12)^\alpha \right) \times B_{1/2}(0)$.
\end{sa}


\begin{sa}[Hölder regularity] \label{thm:hoelder_main} Assume \eqref{eq:K-1}
and \eqref{eq:K-2} hold true for some $\Lambda \geq 1$ and $\alpha \in
(\alpha_0,2)$. There is a
constant $\beta=\beta(d,\alpha_0,\Lambda)$ such that for every solution $u$ of
\eqref{eq:par_equation} in $Q=I \times \Omega$ with $f=0$ and every $Q'\Subset
Q$ the following estimate holds:
\begin{align}\label{eq:hoelder_main}
 \sup_{(t,x),(s,y) \in Q'} \frac{ \bet{u(t,x)-u(s,y)} }{\left( \bet{x-y} +
\bet{t-s}^{1/\alpha} \right)^{\beta}} \leq \frac{\norm{u}_{L^\infty(I \times
\R^d)}}{\eta^\beta} \,, \tag{HC}
\end{align}
with some constant $\eta=\eta(Q, Q') >0$.
\end{sa}

It is worth pointing out the fact that both Theorems
\ref{thm:harnack} and \ref{thm:hoelder_main} only require
(\ref{eq:par_equation}) to hold in some
region $I \times \Omega$, where $\Omega$ may not be the whole space $\real^d$. 
There are only
very few such local results for integro-differential operators in ``divergence
form'' form, see \autoref{subsec:DivergenceForm} for a definition. One of the
standard approaches in variational calulus, proving a local Caccioppoli
inequality, already is nontrivial. Global Caccioppoli inequalities are easier to
obtain, though. Note that in this work we prove and use local Caccioppoli
inequalities for negative and small positive powers of supersolutions only.

Let us briefly comment on the weak Harnack inequality. We recall that
\eqref{eq:harnack-speziell} is called weak
Harnack inequality, cp. \cite[Sec. 1]{Tru68} and would be called Harnack
inequality or strong Harnack inequality if $\norm{u}_{L^1(U_\ominus)}
$ were replaced by $\sup_{U_\ominus} u$. In terms of Hölder
regularity the weak Harnack inequality is as good as the (strong) Harnack
inequality, i.e. both imply the same a priori bounds in Hölder spaces.
However, with regard to other results such as heat kernel bounds one wants to
know when a Harnack inequality holds true. In our general set-up it can easily
be seen that this is not possible. As explained in \autoref{sec:examples} our
assumption allows to study the generator $L$ of a process $X=(X_1,X_2)$ where
$X_i$ are one-dimensional symmetric stable processes. A strong Harnack
inequality cannot hold for such a process, even not in the elliptic case. One
can construct a sequence of functions $u_n:\R^d \to
\R$ which are harmonic in the unit ball $B$ , nonnegative in $\R^d$ but satisfy
$u_n(x_n)/u_n(0) \to +\infty$ for an appropriate sequence $x_n \to 0$. We
remark that such a counterexample exists also in the case where $\mu(x, \d y)$
is of the form $\mu(x, \d y) = j(x-y) \, \d y$ for a specific choice
of the function $j$, see \cite{BoSt05}.

The article is organised as follows: In \autoref{sec:review} we review related
results from the literature on nonlocal parabolic problems. We explain the
notion of nonlocal operators in divergence form and non-divergence form. In
\autoref{sec:examples} we present two examples of measures $\mu(x,\d y)$ which
are not absolutely continuous with respect to the Lebesgue measure in $\R^d$
but satisfy the conditions \eqref{eq:K-1} and \eqref{eq:K-2}. In
\autoref{sec:proofs} we explain the arguments used in the proofs of
\autoref{thm:harnack} and \autoref{thm:hoelder_main}. We concentrate on the
ideas and refer to \cite{FeKa13} for technical details.

\bigskip

{\bf Acknowledgement:} We thank Matthieu Felsinger for several helpful
comments.


\section{A short survey on results for parabolic equations
involving nonlocal operators}\label{sec:review}

For a start let us look at a simple example. Choose $a(t,x,y) \mu(x,\d y) =
|x-y|^{-d-\alpha} dy$, where $\d y$ denotes the Lebesgue measure. For this
choice and for smooth functions $u$, say $u\in C^\infty_c(\R \times \R^d)$, 
$Lu$ equals $-c(d,\alpha) (-\Delta)^{\alpha/2}u$ where we use 
\begin{align}\label{eq:def_frac_lapl}
 - (- \Delta)^{\alpha/2} u (t,x) = \frac{2^{\alpha}
\Gamma(\frac{d+\alpha}{2}
)}{\pi^{d/2} |\Gamma(\frac{-\alpha}{2})|} \lim\limits_{\eps \to 0+}
\int_{|y-x|>\eps} \frac{ u(t,y)-u(t,x)}{ |y-x|^{d+\alpha}} \,\d y \,. 
\end{align}
This definition implies for $v \in C^\infty_c(\R^d)$ 
\begin{align*}
 \widehat{(- \Delta)^{\alpha/2}  v}(\xi) = |\xi|^\alpha
\widehat{v} (\xi) \,. 
\end{align*}
Thus $(-\Delta)^{\alpha/2}$ is a
pseudodifferential operator with symbol $|\xi|^{\alpha}$ which justifies the
use of the symbol and the name ``fractional Laplacian''. In place of
$- (-\Delta)^{\alpha/2}$ we write $ \Delta^{\alpha/2}$. The precise value of
the constant $c(d,\alpha)^{-1}$ in \eqref{eq:def_frac_lapl} is not important but
the following property is: The quantity $c(d,\alpha)^{-1} \alpha
(2-\alpha)$ remains positive and bounded for $\alpha \to 0+$ and $\alpha \to
2-$. 

Thus the equation \eqref{eq:par_equation} reduces in
the most simple case to
\begin{align}\label{eq:heat_eq_frac} 
\partial_t u(t,x) - \Delta^{\alpha/2} u (t,x) = f(t,x), \qquad
(t,x) \in (0,\infty) \times \R^d \,, 
\end{align}
In the case $f = 0$ the solutions can be obtained as
\[ u(t,x) = \int_{\R^d} \Phi^{(\alpha)}(t,x;t_0,y) u_0(y) \,\d y = \int_{\R^d}
p^{(\alpha)}(t,x-y,0) u_0(y) \,\d y \,, \]
when $u_0(x) = u(t_0,x)$, $\Phi^{(\alpha)}$ is the fundamental solution and
$p^{(\alpha)}(t,x,y)= \Phi^{(\alpha)}(t,x;0,y)$. The function $p^{(\alpha)}$ is
often referred to as the heat kernel of the corresponding semi-group. Only in
the case $\alpha =1$ an explicit expression for $p^{(\alpha)}(t,x,y)$ is known.
But one can prove that there is a continuous function $\phi^{(\alpha)}:\R^d \to
(0,\infty)$ which is rotationally invariant and satisfies $c^{-1}
|x|^{-d-\alpha} \leq \phi^{(\alpha)}(|x|) \leq c |x|^{-d-\alpha}$ for some
constant $c \geq 1$ and all $x$ with $|x|>1$ such that for all $t>0$ and all
$x,y \in \R^d$
\begin{align*}
 p^{(\alpha)}(t,x,y) = t^{-d/\alpha} p^{(\alpha)}(1, t^{-1}x, t^{-1}x) =
t^{-d/\alpha} \phi^{(\alpha)}(t^{-1} (x-y)) \,. 
\end{align*}
As a result we obtain for all $t>0$ and all $x,y \in \R^d$
\begin{align}\label{eq:hk_bound_stable}
 p^{(\alpha)}(t,x,y) \asymp t^{-d/\alpha} \left( 1 \wedge
\frac{t^{1/\alpha}}{|x-y|} \right)^{d+\alpha}
\end{align}
Here, the symbol $\asymp$ denotes that the quotient of the two expression
involved stays positive and bounded. 

It is worthwile to compare the situation for nonlocal operators with the one
for local operators. For $\alpha =2$ the nonlocal problem
\eqref{eq:heat_eq_frac} becomes the classical heat equation with  
\begin{align*}
 p^{(2)}(t,x,y) = \frac{1}{(4\pi t)^{d/2}} e^{\frac{-|x-y|^2}{4t}} \,.
\end{align*}
One of the main results in the theory of elliptic differential operators of
second order and in divergence form are the so called Aronson bounds. Every
operator of the form $u \mapsto -\operatorname{div}(A(\cdot) \nabla u)$, where
$A(x)$ are uniformly positive definite matrices possesses a heat kernel $q$
which satisfies the following bounds. There are positive constants $c_1, c_2,
c_3, c_4$ such that for all $t>0$ and all $x,y \in \R^d$
\begin{align*}
 c_1 p^{(2)}(t,c_2|x-y|) \leq q(t,x,y) \leq c_3 p^{(2)}(t,c_4|x-y|) \,.
\end{align*}
This means that the heat kernel for the classical heat equation controls the
heat kernels for nondegenerate parabolic problems. Below we provide details of
how this phenomenon remains true for nonlocal problems.
 
A second result, and this is most important for the study of nonlinear
problems, are a-priori Hölder estimates for solutions to parabolic equations
involving the operator $u \mapsto -\operatorname{div}(A(\cdot) \nabla u)$ where
the dependence of $A(x)=(a_{ij}(x))$ on $x$ is only measurable and
bounded. This part of the theory is due to De Giorgi, Nash and Moser. 

During the last years there have been several attempts to work out a similar
program for nonlocal operators which generalise the fractional Laplacian like
the operator $u \mapsto -\operatorname{div}(A(\cdot) \nabla u)$ generalises the
classical Laplacian. We expand upon this in
\autoref{subsec:DivergenceForm} which discusses divergence form operators.

Many cases involving operators similar to the fractional Laplacian also fall
into the category of non-divergence form operators, and there has also been much
recent progress investigating regularity properties similar to Theorems
\ref{thm:harnack} and \ref{thm:hoelder_main} above.  We mention these results in
\autoref{subsec:NonDivergenceForm}.

\subsection{Divergence Form Operators}\label{subsec:DivergenceForm}
Let us review some results from the literature which
concern regularity issues for solutions to equations similar to
\eqref{eq:par_equation}.  By ``divergence form'', we mean operators
defined as (\ref{eq:def_L}) if the corresponding energy form is well-defined. In
this case solutions are defined with the help
of bilinear forms.  If (\ref{IntroEq:MuSymmetry}) holds, then the bilinear
forms are symmetric. 

Note that all previous results in this direction assume
the measures $\mu(x,\d y)$ to be absolutely continuous with respect to the
Lebesgue-measure on $\R^d$. Thus, for the range of this little survey, we assume
that for some appropriate function $k:(0,\infty) \times \R^d \times \R^d \mapsto
[0,\infty)$ and all $t>0$, $x,y \in \R^d$
\begin{align*}
a(t,x,y) \mu(x,\d y) = k(t,x,y) \,\d x \, \d y\,, \quad k(t,x,y)=k(t,y,x) 
\end{align*}

The method of Nash is applied to equation \eqref{eq:par_equation} by Komatsu in
\cite{Kom88, Kom95}. The author assumes $k(t,x,y)$ to be positive, continuous in
$t$ and to satisfy pointwise bounds of the form $k(t,x,y) \asymp
|x-y|^{-d-\alpha}$ for small values of $|x-y|$. The main results are the
existence and Hölder regularity of a fundamental solution together with some
pointwise estimates. Note that Nash's method is a global one, i.e. one works in
whole of $\R^d$.  

An important contribution to theory is \cite{BaLe02}. Bass and Levin prove
sharp pointwise bounds on the heat kernel and a Harnack inequality. The state
space is $\mathbb{Z}^d$ but this does not constitute a serious limitation.
The upper bounds are proved using Davies' method. The lower bounds are based on
the Harnack inequality which is derived with the help of Krylov-Safonov type
estimates. The existence of a related stochastic process and its
properties are essential for this part of the proof. It is assumed in
\cite{BaLe02} that the function $k(t,x,y)$ is constant in $t$, i.e.
$k(t,x,y)=k'(x,y)$ with $k'(x,y) \asymp |x-y|^{-d-\alpha}$ for all values of
$|x-y|$.   

The results of \cite{BaLe02} are extended in \cite{ChKu03} to a general
$d$-set $(F,\nu)$. Moreover, the theory of Dirichlet forms is applied
consistently which allows to work with the weak formulation as it should be. Chen and
Kumagai show that, under the assumption $k(t,x,y)=k'(x,y)$ with $k'(x,y) \asymp
|x-y|^{-d-\alpha}$ there exists a Feller process that corresponds with the
Dirichlet form $(\cE, \cF)$ given by
\begin{align}
\cE(u,v) &= \iint\limits_{F\, F} \left( u(y) - u(x)\right)\left( v(y) -
v(x)\right) k'(x,y) \nu(\d x) \nu(\d y)\,, \\
\cF &= \Big\{ w \in L^2(F,\nu) |\, \iint\limits_{F\, F} \frac{\left( w(y) -
w(x)\right)^2}{|x-y|^{d+\alpha}} \nu(\d x) \nu(\d y) < +\infty \Big\}\,.
\end{align}
It is shown in \cite{ChKu03} that the heat kernel $p(t,x,y)$ exists and
satisfies  
\begin{align}\label{eq:hk_bound_chen-kumagai}
 p(t,x,y) \asymp t^{-d/\alpha} \left( 1 \wedge
\frac{t^{1/\alpha}}{|x-y|} \right)^{d+\alpha}
\end{align}
for all $x,y \in F$ and $0 < t \leq 1$. This is an Aronson bound for nonlocal
oerators, cp. \eqref{eq:hk_bound_stable}. 

A more general situation concerning the regularity at $|x-y|=0$ is treated by Barlow, Bass, Chen and the first author in
\cite{BBCK09}. Again, $k(t,x,y)$ is assumed to be constant
in $t$, i.e. $k(t,x,y)=k'(x,y)$ but may satisfy  
\begin{align}\label{eq:assum_bbck} 
c_0 |x-y|^{-d-\alpha} \leq k'(x,y) \leq  c_1 |x-y|^{-d-\beta}\,, \quad  |x-y|
\leq 1  
\end{align}
for some positive constants $c_0,c_1,\alpha, \beta$ with $0 < \alpha \leq \beta
< 2$. The existence of a heat kernel $p(t,x,y)$ is shown together with some
upper and lower bounds which hold true outside of some exceptional set. Since
$k$ does not have upper and lower bounds which allow for the same type of
scaling, these results cannot be used in order to deduce regularity results.
Quite contrary to this, \cite{BBCK09} provides an example of a kernel satisfying
\eqref{eq:assum_bbck} and a corresponding discontinuous harmonic function.
Moreover it is shown that the martingale problem is not well posed in general
under the condition \eqref{eq:assum_bbck}.

The situation is better if the exponent $\alpha$ depends on $x$ and has some
regularity. In \cite{BKK10} the main
assumption is 
\begin{align}\label{eq:assum_bkk} 
c_0 |x-y|^{-d-(\alpha(x)\wedge \alpha(y))} \leq k'(x,y) \leq  c_1
|x-y|^{-d-(\alpha(x)\vee \alpha(y))} \,, \quad |x-y|
\leq 1  
\end{align}
for some positive constants $c_0, c_1$ and a Dini-continuous function $\alpha$
taking values in some closed subinterval of $(0,2)$. Bass, Kumagai and the
first author prove bounds of the heat kernel and a priori Hölder
estimates.  

As mentioned in the introduction, the results of this article are based
on the work by Felsinger and the first author in \cite{FeKa13}. Thus we
do not comment on \cite{FeKa13} here in greater detail. Several of the technical
ingredients used in this work and in \cite{FeKa13} are
based on ideas in \cite{Kas09}. The first author establishes a Moser iteration
scheme for elliptic operators corresponding to \eqref{eq:par_equation} leading
to a weak Harnack inequality and Hölder regularity estimates. This article for
the first time establishes a PDE approach to local regularity for nonlocal
operators with measurable coefficients. Moreover, the results are robust, i.e.
the constants in the assertion stay positive and bounded for $\alpha \to 2-$. 
 
In \cite{CCV11} Caffarelli, Chan, and Vasseur prove H\"older regularity for weak
solutions of (\ref{eq:par_equation}) in all of $\real^d$.  The methods follow
the spirit of De Giorgi's
program, but there are non-trivial modifications to account for the nonlocality
of $L$.  Their estimates are global in the sense that all of the quantities
estimated require the equation and the behavior of $u$ to hold globally. Their results
are slightly different than the ones of \cite{FeKa13} in the sense that the
H\"older continuity is controlled in terms of $\norm{u}_{L^2(\real^d)}$
instead of $\norm{u}_{L^\infty(\real^d)}$, and they are not robust as $\al\to2-$.  \cite{CCV11} also provides an important 
justification of how in the nonlocal setting, results such as \cite[Theorem 2.2]{CCV11}, \cite[Theorem 1.2]{FeKa13}, 
and Theorem \ref{thm:hoelder_main} above, also apply to fully nonlinear equations as well as give $C^{1,\al}$ regularity 
for nonlinear, translation invariant operators which could be of the form
\[
F u (t,x)= \operatorname{p.v.}  \int_{\R^d} \phi'\left( u(t,y)-u(t,x)\right)
\mu(x,\d y),
\] 
assuming $F$ is translation invariant (typically shown as $\mu(x,dy)=K(x-y)dy$) and $\phi$ is a convex function satisfying some mild regularity and coercivity assumptions.  Cf. \cite[Section 13]{CaSi-09RegularityByApproximation} and \cite{CaSi-09RegularityIntegroDiff} for the non-divergence elliptic setting.

It is worth mentioning a recent use of regularity results for
equations similar to (\ref{eq:par_equation}) 
to random matrix theory and random Schr\"odinger equations. Erdös and Yau
\cite{ErYa13} apply the ideas of \cite{CCV11} to discrete situations. In our
context their assumption, cf. \cite[Thm 9.1]{ErYa13},
would allow $k(t,x,y)$ to be unbounded as a function of $t$ which is not
included in our approach.  In
\cite{Gome-2012RadiativeTransLongRangeCor}, Gomez uses a priori estimates for a 
hypo-elliptic version of (\ref{eq:par_equation}) for studying the limits of a
random Schr\"odinger equation to a deterministic nonlocal radiative transfer
equation. A very different application of \autoref{thm:harnack} appears in
\cite{JaWe13} where the authors study the asymptotic symmetry of solutions to
nonlinear fractional reaction diffusion equations. In \cite{JiXi13} a Harnack
inequality is used to study a fractional version of the Yamabe flow.

It is important to point out that this little survey is by no means complete.
Several
interesting areas are neglected resp. not treated according to their presence
in journals. One such area is given by heat kernel bounds for more general nonlocal
Dirichlet forms on general state spaces. The articles \cite{ChKu08},
\cite{BBK09}, \cite{CKK11} extend \cite{ChKu03} and treat regularity and
estimates of the heat kernel. For a general introduction to
off-diagonal heat kernel bounds involving non-Gaussian tails, see \cite{GrKu08}.
A detailed study of off-diagonal upper bounds can be found in \cite{BGK09}. A
recent work is \cite{GHL13}. Therein, the authors derive off-diagonal upper
bounds from on-diagonal ones. A key feature is that the walk dimension of the
underlying space is allowed to exceed $2$.

Another area concerns heat kernels for nonlocal operators in bounded
and unbounded domains or equivalently transition densities for jump processes in
such domains. In \cite{CKS10b} sharp two-sided estimates
for the transition density (heat kernel) of the symmetric $\alpha$--stable
process (the fractional Laplacian) killed upon exiting a $C^{1,1}$ open set
$\Omega \subset \R^d$ are obtained. Sharp bounds for the heat kernel up to the
boundary of a bounded domain $\Omega$ in the so called censored case, i.e. when
the integral in \eqref{eq:def_L} is taken over $\Omega$ instead of the full
space, are proved in \cite{CKS10}.

\subsection{Non-Divergence Form Operators}\label{subsec:NonDivergenceForm}

By a ``non-divergence form'' operator in (\ref{eq:def_L}), we mean an operator 
\begin{align}\label{eq:def_Lnondiv}
 L u (t,x)= \int_{\R^d} \left[ u(t,y)-u(t,x)
- \mathbbm{1}_{B_1}(y-x) \big\langle \nabla u(t,x), (y-x) \big\rangle
\right]
 m(t,x,\d y) \,.
\end{align}
for a  general family of measures $m$ such that for all $t$ and $x$,
$m(t,x,dy)$ is still a L\'evy measure in the $y$ variable. A less general and
time-independent version which still is characteristic would
be given by:
\begin{align}\label{eq:def_Lnondiv_simple}
 L v (x)= \int_{\R^d} \left[ v(y)-v(x)
- \mathbbm{1}_{B_1}(y-x) \big\langle \nabla v(x), (y-x)
\big\rangle
\right]
 J(x,y-x) \d y \,,
\end{align}
where $J$ is a function satisfying $\sup_{x\in\R^d} \int \min\{1,|h|^2\}
J(x,h) \,\d h < + \infty$. If additionally $J(x,h)=J(x,-h)$ for all $x$ and
$h\ne 0$, then the operator $L$ from \eqref{eq:def_Lnondiv_simple} reduces
further to 
\begin{align}\label{eq:def_Lnondiv_verysimple}
 L v (x)= \frac12 \int_{\R^d} \left[ v(x+h)-2v(x)+v(x-h) \right]
 J(x,h) \d h \,.
\end{align}
The assumption $J(x,h)=J(x,-h)$ ensures that this operator is of
divergence form with $k(x,y) = J(x,y-x)$ in many cases. Note that the
corresponding energy form needs to be well-defined which forms an additional
restriction. The bilinear form is symmetric if
$J(x,h)=J(x+h,-h)$ for all $x$ and $h\ne 0$ which basically means for the
corresponding process that a jump from $x$ by $h$ to $x+h$ has the same
probability as a jump from $x+h$ by $-h$ back to $x$.  

The second order analog of these nonlocal types of operators are of the form
$a_{ij}(x)u_{x_ix_j}$ as opposed to $\Div(A(x)\grad u(x))$.  The condition
$J(x,h)=J(x,-h)$ for all $x$ and $h\ne 0$ is similar to the symmetry of the
matrices $a_{ij}(x)$. In the second order
case, the main 
program for proving regularity of solutions of equations similar to
(\ref{eq:par_equation}) with estimates that do not depend on the 
$x$-dependence of the equation (the so-called ``bounded measurable coefficients'') goes back to Krylov and Safonov 
\cite{KrSa-1980PropertyParabolicEqMeasurable} where they prove Harnack inequality and H\"older regularity for solutions of the 
second order version of (\ref{eq:par_equation}).  Lihe Wang extends the
Krylov-Safonov program to viscosity solutions (c.f. \cite{CrIsLi-92}) of fully
nonlinear second order versions of (\ref{eq:par_equation}).

A Harnack inequality and Hölder regularity results are obtained for
linear equations by Bass and Levin in \cite{BaLe02b} not assuming regularity of
the kernel with respect to the state variable. Global Schauder estimates are
obtained by Bass in \cite{Bas09} where almost no regularity with respect to the
jump variable is assumed.

The key argument in the PDE approach to local regularity is
sometimes referred to as a
``point-to-measure'' estimate or a ``growth lemma'' (\cite[Lemma
10.1]{CaSi-09RegularityIntegroDiff}, \cite[Lemma
5.6]{ChDa-2012RegNonlocalParabolicCalcVar}, \cite[Lemma 4.1]{Sil06}), the main
consequence of which is what is known as a sort of weak Harnack inequality which
typically carries the name ``$L^\ep$-lemma'' (\cite[Theorem
10.3]{CaSi-09RegularityIntegroDiff}, \cite[Theorem
5.1]{ChDa-2012RegNonlocalParabolicCalcVar}).  
The main tool which is required to begin this program is the
Aleksandrov-Bakelman-Pucci-Krylov-Tso inequality which only requires $u$ to be a
subsolution of (\ref{eq:par_equation}) in a cylinder $Q$ and less than or equal
to zero on the boundary. It gives (cf. \cite[Theorem
3.14]{Wang-1992ParabolicRegularity-one})
\begin{equation}\label{NonDivEq:ABPKT}
\sup_{Q} u \leq \norm{f}_{L^{d+1}(\{u=\Gamma\})},
\end{equation}
where $\Gamma$ is a special envelope of $u$, typically taken to be a concave envelope in $(t,x)$ in the second order 
setting, but can be much more complicated in the nonlocal setting \cite[Theorem
4.1]{ChDa-2012RegNonlocalParabolicCalcVar}.  
In fact, the absence of a nonlocal analog for the Alexsandrov-Bakelman-Pucci-Krylov-Tso inequality as well as an appropriate 
corresponding envelope, $\Gamma$, for the nonlocal setting has been a significant stumbling block for both the parabolic and 
elliptic non-divergence results.  This is especially true when one hopes to prove results which are robust as $\al\to2-$.  
In most of the existing works, a finite cube approximation to the estimate has
been devised as a substitute for (\ref{NonDivEq:ABPKT}). H\"older regularity
results in these situations are derived by Caffarelli and Silvestre
in \cite{CaSi-09RegularityIntegroDiff}. So far the only result in this direction
which avoids coverings is \cite{GuSc-12ABParma} by Guillen and the second
author in the
stationary setting of (\ref{eq:par_equation}). By construction, these
point-to-measure estimates 
are inherently local results, and as a consequence most of the corresponding regularity results along this line of work are local as well.  
This is in contrast to the divergence setting discussed above.

Although it applies to the elliptic setting, a key contribution to the non-divergence theory is the work of Silvestre \cite{Sil06}, 
where he studies the H\"older regularity for solutions of the stationary version of (\ref{eq:par_equation}).  
The results are not robust, but they do not make any sort of special symmetry assumptions on the kernels, and the 
method of proof of the point-to-measure estimate has been a significant inspiration for 
\cite{Silv-2011DifferentiabilityCriticalHJ} and \cite{ChDa-2012RegNonlocalParabolicCalcVar}.  
An interesting and important phenomenon of competing scaling influences happens when (\ref{eq:par_equation}) 
contains a drift term, $b(x)\cdot\grad u$, and the order is $\al=1$.  In this case the gradient term 
(which is non-regularizing) and the nonlocal term, $Lu$, (which is regularizing) have the 
same order-- so a priori it is not clear which influence will be dominant.  
In \cite{Silv-2011DifferentiabilityCriticalHJ}, Silvestre proves that 
indeed the nonlocal term has enough influence to regularize and solutions will
be classical. Note that we do not review the results for the 2-dimensional
dissipative surface quasi-geostrophic equations. These are equations of the form
\eqref{eq:par_equation} including a drift $b$ which is
assumed to be divergence free -- so not quite within the scope of
(\ref{eq:par_equation}). The critical case is $\alpha=1$ and $v_i \in
L^{\infty}(\operatorname{BMO})$, see \cite{KNV07}, \cite{CaVa10} and
\cite{MaMi13}.

Very recently the fully nonlinear, non-divergence version of (\ref{eq:par_equation}) was 
treated by Chang Lara and D\'{a}vila in
\cite{ChDa-2012RegNonlocalParabolicCalcVar}, where they prove robust and 
local H\"older regularity via a hybrid of the methods from the stationary 
case \cite{CaSi-09RegularityIntegroDiff} and a modification of the techniques 
of \cite{Silv-2011DifferentiabilityCriticalHJ}.

There are also some results which use methods other than the ones above.  Related to a Hamilton-Jacobi 
version of (\ref{eq:par_equation}) (similar to the equation treated by
Silvestre), Droniou and Imbert in \cite{DronImbert-2006FractalFirstOrdARMA} and
Imbert in \cite{Imbert-2005NonlocalRegularizationHJ} 
study regularity properties and deduce classical solutions in many situations.  
These results do not recover the critical result of $\al=1$.  Very general local and 
nonlocal fully nonlinear parabolic problems are treated via representation formulas 
from optimal control in
\cite{CardRain-2011HolderRegLocalNonlocalSuperQuadSIAM} Cardaliaguet and Rainer.
They prove that solutions are uniformly H\"older regular with bounds which only 
depend on the relevant extremal operators-- hence only lower and upper bounds, 
but not on the regularity in the coefficients in $x$-- for their version of 
(\ref{eq:par_equation}).  An example of such an equation would be
\[
u_t + H(x,\grad u, D^2u) -Lu = f,
\] 
where $L$ is a non-divergence form of (\ref{eq:def_L}).

\subsection{Related areas}
Let us also mention the currently very active field of
potential theory for subordinate Brownian motion. A model case of a generator
would be given by $L= \phi(-\Delta)$ where $\phi$ is a complete Bernstein
function. In this specific case one can employ several formulas in order to
prove estimates on the heat kernel or the Green function. But if one considers
operators with variable coefficients which generalize $\phi(-\Delta)$, then one
faces several obstacles, the absence of scaling properties being one of them.
Moreover, there are interesting limit cases where the order of
differentiability, which in general is a function and not just a number, is very
close to zero (less than any positive $\eps$) or very close to $2$. We refer the
reader to \cite{KSV12} for an introduction and to \cite{KiMi12}, \cite{CKS13}
for recent contributions concerning regularity. The existence of solutions to
parabolic nonlocal problems in non-divergence form has been studied by
several authors (and long ago). We only mention \cite{AbKa09}\footnote{Note
that one of the main assumptions in \cite{AbKa09}, condition (1.4), contains a
typo, it should be $k_1(x,y) \geq c |y|^{-n-\alpha}$ for small values of $|y|$.}
and \cite{MiPr13} and refer the reader to the references therein.

\subsection{Discussion of \autoref{thm:harnack} and \autoref{thm:hoelder_main}}
Let us underline, in light of the aforementioned articles, what we view to be
the main contributions of the approach laid out in \cite{FeKa13} and extended
in the present work. 

We prove local regularity results such as a weak Harnack inequality. By a
``local regularity result'' we understand an assertion for functions which are
supposed to satisfy the corresponding equation only in a certain bounded set.
The assertion then says something about these functions in the interior of this
set. Harnack or weak Harnack inequalities are protoptypes of such results. 
In the introduction we already mention that our results are robust for
$\alpha \nearrow 2$, i.e. the constants in our main theorems do not depend on
$\alpha \in (\alpha_0,2)$.

\section{Two examples}\label{sec:examples}

The two main examples we have in mind for our results are the singular measure
supported on the axes

\begin{equation}\label{ExamplesEq:Axes}
	\mu_{axes}(x,dy) = \sum_{i=1}^d
\left[\abs{x_i-y_i}^{-1-\al}dy_i\prod_{j\not=i}\delta_{x_j}(dy_j)\right],
\end{equation}
and the measure in $\real^2$ supported on a cusp near $y=0$
\begin{equation}
\mu_{cusp}(x,dy) = \Indicator_{\{\abs{z_2}>\abs{z_1}^s \vee
\abs{z_1}>\abs{z_2}^s \}
}(x-y)\abs{x-y}^{-d-\al}dy.
\end{equation}
Here we have used the notation for $z\in\real^d$, $z=(z_1,z_2)$ and $0<s<1$.

The canonical measure is
\begin{equation}
	\mu_{\al}(x,dy) = \abs{x-y}^{-d-\al}dy.
\end{equation}

The measure $\mu_{axes}$ corresponds to a pure jump process whose coordinates
are given by \emph{independent} one dimensional L\'evy processes of order $\al$.
 As indicated by the product of $\delta_{x_j}$ measures, this measure only
charges differences of $x$ and $y$ which occur on \emph{one} coordinate axis,
and this corresponds to the fact that each jump occurs along one of the
coordinate
directions.  The measure $\mu_{cusp}$ is simply the restriction of the
$\al$-stable measure $\displaystyle \abs{x-y}^{-d-\al}dy$ to the set where $x-y$
fall in the cusp region.

{\bf Remark} It is worth emphasizing that $\mu_{cusp}$ gives rise to an
operator
which is not of $\al$ order, but rather of order 
	\[
	\beta = (1-1/s) + \al < \al \quad \text{ if } \beta > 0 \,.
	\]
The lower order comes from the fact that the cusp causes spheres of
small radius to be charged with a larger power of $r$ than in the case with the
canonical measure, $\mu_\al$. Not surprisingly this places a restriction on the
possible choices for $s$ once $\al$ is fixed.

In the remainder of this section, we check the assumptions for these two
important example measures.  It will be useful to have some notation.

\begin{equation}
	\E_{\mu,A,B}(u,v) := \int_A\int_B (u(x)-u(y))(v(x)-v(y))\mu(x,dy)dx
\end{equation}

\begin{align}
	\E^{\al} &= \E_{\mu}\ \text{for}\ \mu(x,dy) = \abs{x-y}^{-d-\al}dy\\
	\E^{axes} &= \E_{\mu}\ \text{for}\ \mu(x,dy) = \mu_{axes}\\
	\E^{cusp} &= \E_{\mu}\ \text{for}\ \mu(x,dy) = \mu_{cusp}
\end{align}

\begin{lem}\label{Examples:MuAxesK1}
	$\displaystyle \mu_{axes}$ satisfies (K1).
\end{lem}

\begin{proof}[Proof of (\ref{eq:K-1}) for $\mu_{axes}$]
	This will follow from the simple observation that the two measures,
$\mu_{axes}$ and $\mu_\al$, charge any annulus centered at $x$ with a value
which only differs by a dimensional constant.  Indeed given any $r<R$ and
$\Rscript(x) = B_R(x)\setminus B_r(x)$,

\begin{align*}
		\mu_{axes}(x,\Rscript(x)) &= \int_{\Rscript(x)} \sum_{i=1}^d
\left[\abs{x_i-y_i}^{-1-\al}dy_i\prod_{j\not=i}\delta_{x_j}(dy_j)\right] = 2
\sum_{i=1}^d \int_r^R \abs{y_i}^{-1-\al}dy_i
\\& =\frac{2d}{\al}(r^{-\al}-R^{-\al}),
	\end{align*}
	and
	\begin{align*}
		\mu_\al(x,\Rscript(x)) &= \int_{\Rscript(x)}
\abs{x-y}^{-d-\al}dy = \frac{C(d)}{\al}(r^{-\al}-R^{-\al}),
	\end{align*}
	where $C(d)$ is the surface area of the $d-1$ sphere in this case.
	
	Now we see that writing both integrals of (\ref{eq:K-1}) using e.g.
dyadic rings makes the calculation equivalent to that of the canonical measure,
$\mu_\al$, which gives the desired result.
\end{proof}

\begin{lem}\label{Examples:MuCuspK1}
	$\displaystyle \mu_{cusp}$ satisfies (K1).
\end{lem}

\begin{proof}[Proof of (\ref{eq:K-1}) for $\mu_{cusp}$]
	This follows almost exactly the argument of the case of $\mu_{axes}$,
except here we can actually work with the surface measure on $\partial B_r(x)$,
since $\mu_{cusp}$ is absolutely continuous with respect to Lebesgue measure. It
is easiest to note that we can represent the contribution of
	\begin{equation}
		\int_{\partial B_r(x)}\Indicator_{\{\abs{z_d}>\abs{Z}^s\}
}(x-y)\abs{x-y}^{-d-\al}\d S_r(y)
	\end{equation}
	as a fraction of the contribution 
	\begin{equation}
		\int_{\partial B_r(x)}\abs{x-y}^{-d-\al} \d S_r(y).
	\end{equation}
	If we let $m_d(\theta)$ be the proportion of the sphere which is within
an angle of $\theta$ with the positive $z_2$-axis we see that the measure
$\mu_{cusp}$ only allows for 
a contribution from a small portion of the sphere depending on $r$.  We have

	\begin{equation*}
		\int_{\partial B_r(x)}\Indicator_{\{\abs{z_d}>\abs{Z}^s\}
}(x-y)\abs{x-y}^{-d-\al}\d S_r(y) = m_d(\theta(r))	\int_{\partial
B_r(x)}\abs{x-y}^{-d-\al} \d S_r(y),
	\end{equation*}
for 
	\begin{equation*}
		\theta(r) =
\sin^{-1}\left(\frac{z_1(r)}{\sqrt{z_1(r)^{2s}+z_1(r)^2}}\right).
	\end{equation*} 
Here we have used the notation that given any fixed $r$,
$z_1(r)$ is the unique positive solution of
	\begin{equation*}
		r^2=z_1(r)^2 + z_1(r)^{2s}.
	\end{equation*}
	The most important feature of this relationship is that for $r$ small
(which can be quantified in terms of $s$), 
	\begin{equation*}
		\abs{z_1(r)} \approx r^{1/s}
	\end{equation*}
	
	The key observations which both identifies the correct fractional scale
of the energy associated to $\mu_{cusp}$ and the fact which allows us to verify
(\ref{eq:K-1}) are that 
	for $r$ small (uniformly, depending only on $d$ and $s$),
	\begin{equation*}
		m_d(\theta(r)) \approx C r^{(1-s)/s},
	\end{equation*}
	and for $r$ large,
	\begin{equation*}
		m_d(\theta(r))\to 1.
	\end{equation*}
	
	The small $r$ behavior of $\mu_{cusp}$ is the main determining factor
for the true scale of 
the operator in (\ref{eq:def_L}).  Indeed we see, for example, that for $r<<1$
	\begin{equation*}
		\int_{B_r}\abs{z}^2\mu(0,dz) \approx C\int_0^r  \tau^{(1-s)/s}
\int_{\partial B_\tau} \abs{z}^{2-d-\al}\d S_\tau(z)
		= \int_0^r \int_{\partial B_\tau} \abs{z}^{2-d-\al-(s-1)/s}\d
S_\tau(z)
	\end{equation*}
	This tells us that $\mu_{cusp}$ does not have scale $\al$, but rather
satisfies (\ref{eq:K-1}) and (\ref{eq:K-2}) with the scale
	\begin{equation*}
		\beta = (1-\frac{1}{s}) + \al < \al\,,
	\end{equation*}
as long as $\beta$ is positive. These heuristic
arguments can be made rigorous in checking
(\ref{eq:K-1}) by choosing a cutoff, $r_0$, which depends only on a lower bound
of $s$, say $s_0\leq s$, such that 
	\[
	\frac{1}{2}\theta(r)\leq \sin(\theta(r))\leq \theta(r)
	\]
	and
	\[
	\sin(\theta(r))\leq \frac{1}{2} z_1(r)^{1-s}\leq
\frac{1}{2}r^{(1-s)/s},
	\]
	for all $r\leq r_0$.  Then the integrals in (\ref{eq:K-1}) can be split
at $r_0$ and estimated using either 
	\[
	m(\theta(r))\leq C(d,s_0)r^{(1-s)/s}
	\]
	or
	\[
	1\geq m(\theta(r))\geq \tilde C(d,s_0).
	\]
\end{proof}

\begin{lem}\label{ExamplesLem:MuAxesK2}
	$\displaystyle \mu_{axes}$ satisfies (K2).
\end{lem}

\begin{proof}[Proof of (\ref{eq:K-2}) for $\displaystyle \mu_{axes}$]
	We include the details for the case $d=2$, and then comment on the few
modifications for the general dimensional case at the end.  So temporarily
assume $d=2$.
	
	First we note that it suffices to check the inequality for $v\in
C^{0,1}(B)$.  This assumption on $v$ allows us to both ignore a small
neighborhood of the diagonal of $B\times B$ and to write all of the energies in
terms of Riemann sums.  This will make the accounting of terms easier.  We use
the following notation
	\begin{align}
		N &\ \text{fixed, depending only on}\ v\\
		\Delta h &= \frac{2\rho}{N}\ \text{fixed, depending only on}\
v\\
		\Delta y &= \Delta x = \Delta h\\
		\Delta r &= \Delta h\\
		\Delta \theta &= \frac{\Delta h}{2\pi\rho}\\
		\{(x_j,y_k)\} &= \text{centers of the cells of}\ (\Delta h
\integer)^2\times (\Delta h \integer)^2 \intersect B\times B\\
		x_j\sim y_k &\ \text{if either}\ (x_j)_1=(y_k)_1\ \text{or}\ 
(x_j)_2=(y_k)_2\\
		r_{jk} &= \abs{x_j-y_k}\\
		v_{jk} &= (v(x_j)-v(y_k))\\
		\Rscript_l(x_j) &= \left\{y_k \in (B_{l\Delta r}(x_j)\setminus
B_{(l-1)\Delta r}(x_j)\intersect B) \right\}\\
	\hat y_{lm} &= y_k\ \text{if the}\ l\Delta r, m\Delta \theta\
\text{(polar) cell lies a majority within}\ Q_{\Delta y}(y_k)  
	\end{align}
	Here we include all of $\Delta x$, $\Delta y$, $\Delta r$, $\Delta h$,
$\Delta \theta$ even though 
we set equality $\Delta y = \Delta x = \Delta h = \Delta r$ and $\Delta \theta =
\Delta h / (2\pi\rho)$.  
This is just to try to indicate which integrals we are discretizing in various
lines of the calculation.  
The main discretization parameter is $\Delta h$, and this is chosen based only
on the $C^{0,1}$ norm of $v$ 
to ensure that all approximations of the various integrals are within a given
tolerance of the continuous versions.

	It will be convenient to approximate the original energy in cartesian
coordinates,
	
	\begin{align*}
		\E^{axes}_{approx} &= \sum_{x_j}\sum_{y_k\sim x_j}
\frac{v_{jk}^2}{r_{jk}^{1+\al}}\Delta y (\Delta x)^2,
	\end{align*}
	the canonical energy in both polar coordinates in the inside integration
	\begin{align*}
		\E^\al_{approx} = \sum_{x_j}\sum_{l}\sum_{y_k\in\Rscript_l(x_j)}
\frac{v_{jk}^2}{r_{jk}^{2+\al}}r_{jk}\Delta r \Delta \theta (\Delta x)^2,
	\end{align*}
	and finally the canonical energy in cartesian coordinates
	\begin{align*}
		\E^\al_{approx} = \sum_{x_j}\sum_{y_k}
\frac{v_{jk}^2}{r_{jk}^{2+\al}}(\Delta y)^2 (\Delta x)^2.
	\end{align*}
	We have chosen to work with a uniform grid in cartesian coordinates,
with cells centered at $(x_j,y_k)\in B\times B$. Since $v$ is fixed and in
$C^{0,1}$, the grid size can be chosen small enough so that this collection of
cartesian grid points
$\{(x_j,y_k)\}$ is sufficient to resolve both the cartesian and polar
approximate decompositions.  This is relevant because the polar approximation
grid does not match the cartesian one, 
but if $\Delta h$ is small enough the cartesian points can be used for the polar
ones without adding significantly to the overall error.

	\noindent
	\textbf{Step 1:} $\displaystyle \E^{axes}_{B,B} \geq C(d)
\E^\al_{B,B}$.

	The strategy is that we can link $\E^{axes}$ with $\E^\al$ via an
intermediate energy which has the same weights in $r_{jk}$ and $\Delta y$ as
$\E^{axes}$ but sums over all $x_j$ and $y_k$ instead of only $y_k\sim x_j$:
	
	\begin{align}\label{ExamplesEq:IntermediateEng}
		\F(v) =
\sum_{x_j}\sum_{y_k}\frac{(v(x_k)-v(y_k))^2}{r_{jk}^{1+\al}} \Delta y (\Delta
x)^2
	\end{align} 
	
The inequality
	\begin{align}
		2\left((v(x)-v(y))^2 + (v(z)-v(y))^2\right) \geq
\left(v(x)-v(y)\right)^2
	\end{align}
allows us to recover the contribution between any two points, $x$ and $y$ from
that of points which satisfy $x\sim y$.  Specifically, given $x$ and $y$, the
choice of $z=(x_1,y_2)$ allows for such a connection.  In particular, for $z$
fixed, all of the entries on the axes passing through $z$ will need to be used
$CN$ ($C$ is dimensional) times in connecting all of the $x$ and $y$ lying on
the axes passing though $z$.  This gives the inequality

\begin{align}
	4N \E^{axes}_{approx} &= 4N\sum_{x_k}\sum_{y_k\sim x_j}
\frac{v_{jk}^2}{r_{jk}^{1+\al}}\Delta h (\Delta h)^2\\
	&\geq \sum_{x_j}\sum_{y_k} \frac{v_{jk}^2}{r_{jk}^{1+\al}}\Delta h
(\Delta h)^2 = \F(v).
\end{align}

\noindent
Next we work from $\F(v)$ to $\E^\al_{approx}$ by using polar coordinates for
the inside integral.  The reason for initially discretizing in polar is to more
easily see the balance between the unmatched powers of $r_{jk}$ and $\Delta y$
between $\E^{axes}$ and $\E^\al$.

	\begin{align}
		\F(v) &=  \frac{1}{\Delta h}\sum_{x_j}\sum_{y_k}
\frac{v_{jk}^2}{r_{jk}^{1+\al}r_{jk}}r_{jk}\Delta h \Delta h (\Delta h)^2\\
		&= (C(d))\frac{1}{\Delta
h}\sum_{x_j}\sum_{\Rscript_l(x_j)}\sum_{y_k\in\Rscript_l(x_j)}
\frac{v_{jk}^2}{r_{jk}^{1+\al}r_{jk}}r_{jk}\Delta r (2\pi \rho\Delta \theta)
(\Delta x)^2\\
		&= (C(d))\frac{2\pi}{\Delta
h}\rho\sum_{x_j}\sum_{\Rscript_l(x_j)}\sum_{y_k\in\Rscript_l(x_j)}
\frac{v_{jk}^2}{r_{jk}^{2+\al}}r_{jk}\Delta r \Delta \theta (\Delta
x)^2\\
		&= C(d)\frac{2\pi N}{\rho} \rho\sum_{x_j}\sum_{y_k}
\frac{v_{jk}^2}{r_{jk}^{2+\al}}(\Delta y)^2 (\Delta
x)^2\\
		&= C(d) N \E^\al_{approx}
	\end{align}

Again, due to the $C^{0,1}$ nature of $v$, $N$ and $\Delta h$, can be chose
appropriately so that we can deduce
\begin{equation}
	\E^{axes}_{B,B} \geq C(d) \E^\al_{B,B},
\end{equation}
which completes step 1.

Now we move on to the easier step

\noindent
\textbf{Step 2:}  $\displaystyle \E^{axes}_{B,B}(v,v) \leq C(d)
\E^\al_{B,B}(v,v)$.

We again appeal to the intermediate energy (\ref{ExamplesEq:IntermediateEng}). 
We reuse the above calculation, where instead we make use of the inequality
\begin{equation}
	(v(x)-v(z))^2 + (v(z)-v(y))^2 \leq (v(x)-v(y))^2.
\end{equation}

\begin{align}
	4N\E^{axes}_{approx} &= \sum_{x_j}\sum_{y_k\sim x_j}
\frac{v_{jk}^2}{r_{jk}^{1+\al}}\Delta y (\Delta x)^2\\
	&\leq \sum_{x_j}\sum_{y_k}\frac{(v(x_k)-v(y_k))^2}{r_{jk}^{1+\al}}
\Delta y (\Delta x)^2\\
	&= C(d)N \E^\al_{approx}.
\end{align}
This completes the argument for $d=2$.

For an arbitrary dimension, the main modification is in the number of times a
given line will be used as auxiliary points to connect arbitrary $x_j$ and $y_k$
on the grid.  In this case, each $y_k$ in $\E^{axes}_{approx}$ will be used
$CN^{d-1}$ times to make these connections.  Matching this power of $N^{d-1}$ is
the term $r^{d-1}(\Delta h)^{d-1}$ which must be introduced to $\F(v)$ to be
compatible with the approximation of $\E^\al$ in spherical coordinates.  Thus we
end up in the same situation (neglecting dimensional constants) with a
multiplication by $N^{d-1}$ on both sides of the inequality.
\end{proof}

It turns out to be at least a little bit non-trivial to check (\ref{eq:K-2}) for
the measure, $\mu_{cusp}$.  Instead of giving a Riemann Sum argument as for
$\mu_{axes}$, we use the fact that (\ref{eq:K-2}) has already been verified in
\cite{DyKa13}.  We note that none of the constants in the proof of
\cite{DyKa13} depend $\rho$ in the underlying set, $B$.
\begin{lem}[\cite{DyKa13}]\label{Examples:MuCuspK2}
	$\displaystyle \mu_{cusp}$ satisfies (K2).
\end{lem}


\section{Proofs}\label{sec:proofs}

In this section we explain how to prove \autoref{thm:harnack} and
\autoref{thm:hoelder_main}. We closely follow the strategy of \cite{FeKa13} and
restrict our explanations to those parts of the proofs which are different
for our set-up.

Let us first explain what it means that a
function $\R \times \R^d \to \R$ satisfies
equation \eqref{eq:par_equation} in a bounded region in the weak sense.  For
$\alpha \in (0,2)$ the Sobolev space of fractional order $\alpha/2$ is defined
by
\begin{align}
H^{\alpha/2}(\Omega) &= \left\{ v \in L^2(\Omega)\colon \norm{v}_{H^{\alpha/2}(\Omega)} < +\infty \right\}
\intertext{where the norm is defined by}
\norm{v}_{H^{\alpha/2}(\Omega)}^2 &= \norm{v}^2_{L^2(\Omega)} + \akon
\iint\limits_{\Omega\, \Omega} \frac{\bet{v(x)-v(y)}^2}{\bet{x-y}^{\alpha+d}} \d x\d y.
\end{align}
We denote by $H^{\alpha/2}_0(\Omega)$ the completion of
$C_c^\infty(\Omega)$
under $\norm{\cdot}_{H^{\alpha/2}(\R^d)}$ and by $H^{-\alpha/2}$ the dual of $H_0^{\alpha/2}$. By $\inf v$ and $\sup v$ we denote 
the essential infimum and the essential supremum, respectively, of a given funktion $v$.
\subsection{Weak solutions}
We define a nonlocal bilinear form associated to $L$ by 
\begin{equation}
 \cE_t(u,v) = \iint\limits_{\R^d\, \R^d} \left[ u(t,y) - u(t,x)\right] \left[
v(t,y)-v(t,x) \right] a(t,x,y)\mu(x,\d y) \d x.
\end{equation}
Let us recall the notion of local solutions from \cite{FeKa13}.
\begin{defi} \label{defi:solution}
Assume $Q=I\times \Omega \subset \R^{d+1}$ and $f \in L^\infty(Q)$. We say that
$u \in L^\infty(I;L^\infty(\R^d))$ is a \emph{supersolution of \eqref{eq:par_equation} in $Q=I \times \Omega$}, if
\begin{enumerate}[(i)]
 \item $u \in C_{loc}(I;L^2_{loc}(\Omega)) \cap L^2_{loc}(I;H^{\alpha/2}_{loc}(\Omega))$,
 \item for every subdomain $\Omega' \Subset \Omega$, for every subinterval $[t_1,t_2]\subset I$ and for every 
nonnegative test function $\phi \in H^1_{loc}(I;L^2(\Omega')) \cap L^2_{loc}(I;H_0^{\alpha/2}(\Omega'))$,
\begin{align} \label{eq:def-solution}
\int_{\Omega'} \phi(t_2,x) u(t_2,x) \d x -&\int_{\Omega'} \phi(t_1,x) u(t_1,x)
\d x -
\int_{t_1}^{t_2} \int_{\Omega'} u(t,x) \partial_t \phi(t,x) \d x \d t +
\int_{t_1}^{t_2} \cE_t(u,\phi) \d t \nonumber \\ 
&\geq \int_{t_1}^{t_2} \int_{\Omega'} f(t,x) \phi(t,x) \d x \d t .
\end{align}
\end{enumerate}
\end{defi}

Instead of writing that $u$ is a supersolution in $I \times \Omega$ we write $\partial_t u - Lu \geq f$ in $I \times \Omega$. 
Subsolutions and solutions are defined accordingly. Working rigorously with local weak solutions for parabolic 
equations requires some care. Instead of the definition above we will use the inequality
\begin{equation}\label{eq:steklov-solution}
 \int_{\Omega'} \partial_t u(t,x) \phi(t,x) \d x + \cE_t(u(t,\cdot),\phi(t,\cdot)) \geq \int_{\Omega'} f(t,x) \phi(t,x)\d x \qquad \text{for a.e. } t \in I, \tag{\ref*{eq:def-solution}'}
\end{equation}
for test functions of the form $\phi(t,x)=\psi(x) u^{-q}(t,x)$, $q >0$, where $u$ is a positive supersolution 
in $I \times \Omega$ and $\psi$ a suitable cut-off function. In particular, we assume that $u$ is a.e. differentiable in time. The use (\ref{eq:steklov-solution}) can be justified using Steklov averages. 

A second crucial ingredient in our  proof is the scaling property of
\eqref{eq:par_equation} which we are now going to explain.

\subsection{Scaling}
We will often make use of scaling. Let us explain how the operator under consideration behaves with respect to rescaled functions. Define $B_r(x_0) =\left\{ x \in \R^d \colon \bet{x-x_0} < r \right\}$ and
\begin{align*}
I_r(t_0)&=(t_0-r^\alpha,t_0+r^\alpha),&Q_r(x_0,t_0)&= I_r(t_0) \times B_r(x_0),&\\
I_{\oplus}(r) &= (0,r^\alpha),& Q_\oplus(r) &= I_{\oplus}(r) \times B_r(0),\\
I_{\ominus}(r) &= (-r^\alpha,0),& Q_\ominus(r) &= I_{\ominus}(r) \times B_r(0).
\end{align*}

\begin{lem}[Scaling property] \label{lem:scaling} 
Let $\alpha \in (0,2)$, $\xi \in \R^d$, $\tau \in \R$ and $r>0$. Define a diffeomorphism $J$ by $J(x)= rx + \xi$. 
Let $u$ be a supersolution of the equation \eqref{eq:par_equation} in some cylindrical domain $I\times \Omega = Q \Supset Q_r(\xi,\tau)$. Define new functions $\widetilde u$ by $\widetilde u(t,x)= u(r^\alpha t+\tau,J(x))$ and $\widetilde f$ by $\widetilde f(t,x)=r^\alpha f(r^\alpha t+\tau,r x+\xi)$. \\ Then $\widetilde u$ satisfies $\partial_t u - \widetilde{L}u \geq \widetilde{f}$ in $\tilde I\times J^{-1}(\Omega) \Supset (-1,1) \times B_1(0)$, where $\widetilde{L}$ is defined as in \eqref{eq:def_L} with $\widetilde{a}(t,x,\d y) \widetilde{\mu}(x,\d y)$ defined by 
\begin{align}
\widetilde{a}(t,x,y) = a(r^\alpha t+\tau,rx+\xi,ry+\xi), \quad \widetilde{\mu} (x,\d y) = r^\alpha (\mu \circ J)(Jx, \d y) \,. 
\end{align}
where $(\mu \circ J)(z, A) = \mu(z, J(A))$. The function $\widetilde{a}$ is positive, bounded and symmetric as is $a$. 
The family of measures $\widetilde{\mu} (\cdot,\d y)$ satisfies assumptions \eqref{eq:K-1} and \eqref{eq:K-2} with the same constants as in the case of $\mu (\cdot,\d y)$.
\end{lem}

\begin{proof}
	Since the coefficient, $a$, is only modified by a straightforward change of variables, we take for simplicity $a(t,x,y)=1$.  
A general $a$ does not modify the argument, but just adds extra notation.
	
	Now we must verify (\ref{eq:def-solution}) for $\tilde u$ in $(-1,1)\times B_1(0)$.
	We start with a set, $I'\times \Omega'\subset (-1,1)\times B_1(0)$ and a test function, 
$\phi \in H^1_{loc}(I';L^2(\Omega')) \cap L^2_{loc}(I';H_0^{\alpha/2}(\Omega'))$, for the $\tilde u$ equation in $\tilde I\times J^{-1}(\Omega) \Supset (-1,1) \times B_1(0)$.  
The definition of $\tilde\mu$ is exactly chosen so that (\ref{eq:def-solution}) reduces, after a change of variables $(s,z)=(r^\alpha t+\tau,J(x))$, to
	
	 \begin{multline}\label{ScalingEq:DefForTildeu}
	\int_{J(\Omega')} \phi(s_2,J^{-1}(z)) u(s_2,z) r^{-d}\d z -\int_{J(\Omega')} \phi(s_1,J^{-1}(z)) u(s_1,z) r^{-d}\d z\\ -
	\int_{s_1}^{s_2}\int_{J(\Omega')} u(s,z) r^\al\partial_s \phi(s,J^{-1}z) r^{-d}\d x r^{-\al}\d t +\\
	+\int_{s_1}^{s_2}\int_{\real^d}\int_{\real^d} \left(u(s,z)-u(s,w)\right)\left(\phi(s,J^{-1}(z))-\phi(s,J^{-1}(w))\right)r^{\al}\mu(z,\d z)r^{-d}\d w r^{-\al}\d s\\
	\geq  \int_{s_1}^{s_2}\int_{J(\Omega')} r^\al f(s,z) \phi(s,J^{-1}(z)) r^{-d}\d z r^{-\al}\d s .
	\end{multline}
Since $r^{-d}\phi(r^\al(\cdot) + \tau, J^{-1}(\cdot))$ is a valid test function in $[s_1,s_2]\times J(\Omega')\subset Q_r(\xi,\tau)$, 
we conclude by the supersolution property of $u$ that (\ref{ScalingEq:DefForTildeu}) holds true.

We also briefly verify that $\tilde \mu$ satisfies (\ref{eq:K-1}), (\ref{eq:K-2}) with the same constants as $\mu$.  
First we verify (\ref{eq:K-1}).  Let $\rho>0$ be fixed, and without loss of generality we check the conditions 
at $x_0=0$ as any other $x_0$ follows the same argument mutatis mutandi.  We calculate the two integrals with the 
change of variables $y=J^{-1}(z)$.  This gives
\begin{align}\label{ScalingEq:K1Int1}
	\rho^{-2}\int_{B_\rho}\abs{y}^2\widetilde \mu(0,dy) &= \rho^{-2}\int_{J(B_\rho)}\abs{J^{-1}(z)}^2 r^\al \mu(J(0),dy)\\
	&= \rho^{-2}\int_{B_{r\rho}(\xi)}\abs{z-\xi}^2 r^{-2}r^\al \mu(\xi,dy)
\end{align}
and
\begin{align}\label{ScalingEq:K1Int2}
	\int_{(B_\rho)^C}\widetilde \mu(0,dy) &= \int_{J((B_\rho)^C)} r^\al \mu(J(0),dy)\\
	&= \int_{(B_{r\rho}(\xi))^C}r^\al \mu(\xi,dy).
\end{align}
Therefore after this change of variables, combined with (\ref{eq:K-1}) for $\mu$, we have
\begin{align}
	&\rho^{-2}\int_{B_\rho}\abs{y}^2\widetilde \mu(0,dy) + \int_{(B_\rho)^C}\widetilde \mu(0,dy) \\
	&\ \ \ \ \ =r^\al(r\rho)^{-2}\int_{B_{r\rho}(\xi)}\abs{z-\xi}^2 \mu(0,dy) + r^\al\int_{(B_{r\rho}(\xi))^C} \mu(0,dy) 
	\leq r^\al \Lam (r\rho)^{-\al}, 
\end{align}
and hence (\ref{eq:K-1}) holds for $\widetilde \mu$.

Next we turn to (\ref{eq:K-2}).  We change variables for the $\widetilde \mu$ energy as $x=J^{-1}(w)$ and $y=J^{-1}(z)$. 
 Then, owing to the fact that $\mu$ satisfies (\ref{eq:K-2}) in the set $J(B)$ with the test function $r^{-d/2}v\circ J^{-1}$, 
we have the validity of (\ref{ScalingEq:K2Int1})--(\ref{ScalingEq:K2Int3}) below.

\begin{align}
	&\Lambda^{-1}  \iint\limits_{B\,B} \left[ v(x)-v(y)\right]^2 \widetilde\mu(x,\d y)  \d x\\ 
	&\ \ \ \ \ = \Lambda^{-1}  \iint\limits_{J(B)\,J(B)} \left[ v(J^{-1}(w))-v(J^{-1}(z))\right]^2  r^\al\mu(J(J^{-1}(w)),\d(J^{-1}(z)) ) r^{-d} \d w\\
	&\ \ \ \ \ = \Lambda^{-1}  \iint\limits_{J(B)\,J(B)} \left[ r^{-d/2}v(J^{-1}(w))-r^{-d/2}v(J^{-1}(z))\right]^2  r^\al\mu(w,\d z) \d w \label{ScalingEq:K2Int1}\\
	&\ \ \ \ \ \leq (2-\al) \iint\limits_{J(B)\,J(B)} \left[ r^{-d/2}v(J^{-1}(w))-r^{-d/2}v(J^{-1}(z))\right]^2 \abs{w-z}^{-d-\al} r^\al \d z \d w \label{ScalingEq:K2Int2}\\
	&\ \ \ \ \ \leq \Lambda \iint\limits_{J(B)\,J(B)} \left[ r^{-d/2}v(J^{-1}(w))-r^{-d/2}v(J^{-1}(z))\right]^2  r^\al\mu(w,\d z) \d w\label{ScalingEq:K2Int3}\\
	&\ \ \ \ \ = \Lambda \iint\limits_{B\,B} \left[ v(x)-v(y)\right]^2 \widetilde\mu(x,\d y)  \d x.
\end{align}
Hence $\widetilde\mu$ satisfies (\ref{eq:K-2}).
\end{proof}

Assumption \eqref{eq:K-2} assures that we can apply
standard Sobolev and Poincar\'{e} embeddings. We also need a weighted
Poincare\'{e} inequality of the following form:

\begin{prop}[Weighted Poincaré inequality] \label{lem:poincare} Let $\Psi \colon
{B_{3/2}} \to [0,1]$ be defined by $\Psi(x)=(\tfrac32-\bet{x}) \wedge 1$.
Then there is a positive constant $c_2(d,\alpha_0, \Lambda)$ such that for every
$v \in L^1(B_{3/2}, \Psi(x)\d x)$
\begin{equation*}
 \int_{B_{3/2}} \left[ v(x)-v_{\Psi} \right]^2 \Psi(x) \d x 
 \leq c_2 \iint\limits_{B_{3/2}\, B_{3/2}} \left[ v(x)-v(y)\right]^2 
\left( \Psi(x) \wedge \Psi(y) \right) \mu(x, \d y) \d x\, ,
\end{equation*}
where $\displaystyle v_{\Psi}=  \Bigl(\int_{B_{3/2}} \Psi(x)\d x\Bigr)^{-1}
\int_{B_{3/2}} v(x) \Psi(x)\d x$.
\end{prop}
The result is a simple application of \cite[Theorem 1]{DyKa12}. A more
involved proof for a smaller class of weights and not being robust for
$\alpha \to 2-$ can be found in \cite{CKK08}.

\subsection{Moser iteration}

The main idea of the proof is to use the lemma of Bombieri-Giusti and to apply
the Moser iteration in order to fulfill the conditions in the lemma. So we need
to iterate positive and negative powers of the supersolution. Let us first look
at negative powers. The following proposition is established by using as a
test function $\phi(t,x) =
\widetilde
u^{-q}(t,x) \psi^{q+1}(x)$ where $q>1$ and $\psi\colon \R^d \to [0,1]$ is
defined by $\psi(x)= \left(
\frac{R-\bet{x}}{R-r} \wedge 1 \right) \vee 0$.

\begin{prop}\label{prop:step-neg} Let $\frac12 \leq r <R \leq 1$ and $p>0$. Then
every
nonnegative supersolution $u$ in $Q = I \times \Omega$, $Q \Supset
Q_\ominus(R)$, with $u \geq \epsilon>0$ in $Q$
satisfies the following inequality
\begin{align} 
  \left( \int_{Q_\ominus(r)} \widetilde u^{-\kappa p}(t,x) \d
x \d t \right)^{1/\kappa} &\leq A  \int_{Q_\ominus(R)} \widetilde u^{-p}(t,x) \d
x \d t\ ,\label{eq:moment-neg}
\end{align}
where $\widetilde u = u + \norm{f}_{L^\infty(Q)}$, $\kappa=1+\frac{\alpha}{d}$
and $A$ can be chosen\footnote{Note that in the cases $d\in\{1,2\}$ one would
need to adjust the definition of $\kappa$.} as
\begin{equation} \label{eq:const-neg}
 A = C(p+1)^2 \left( \left(R-r\right)^{-\alpha}+ (R^{\alpha}-r^\alpha)^{-1}
\right) \qquad \text{with } C=C(d,\alpha_0,\Lambda).
\end{equation}
\end{prop}
The only difference in the proof to the version in \cite{FeKa13} is
the following: This time, rather than estimating  
\begin{align*}
4 \iint\limits_{B_{R}\, B_R^c} \left[ \psi(x)-\psi(y)\right]^2
\widetilde u^{1-q}(t,x) k_t(x,y) \d y \d x
\end{align*}
from above, we need to estimate
\begin{align*}
4 \iint\limits_{B_{R}\, B_R^c} \left[ \psi(x)-\psi(y)\right]^2
\widetilde u^{1-q}(t,x) a(t,x,y) \mu(x, \d y) \d x \leq c_1(d,\Lambda)
(R-r)^{-\alpha}
\int_{B_R} u^{1-q}(t,x) \d x,
\end{align*}
which is possible due to assumption \eqref{eq:K-1} and properties of $\psi$.
Once, \autoref{prop:step-neg} is established, we can iterate the result and
obtain:
\begin{sa}[Moser iteration for negative exponents] \label{thm:inf-est} Let
$\frac12\leq r<R\leq
1$ and $0<p\leq1$. Then there is a constant $C=C(d, \alpha_0, \Lambda)>0$
such that for every nonnegative supersolution $u$ in $Q = I \times \Omega$, $Q
\Supset Q_\ominus(R)$, with $u \geq
\epsilon >0$ in $Q$ the following estimate holds:
\begin{align}
\sup_{Q_{\ominus}(r)} \widetilde u^{-1} &\leq
\left(\frac{C}{G_1(r,R)}\right)^{1/p} \left( \int_{Q_{\ominus}(R)} \widetilde
u^{-p}(t,x) \d x \d t \right)^{1/p}, \label{eq:inf-est}
\end{align}
where $\widetilde u = u + \norm{f}_{L^\infty(Q)}$ and 
$G_1(r,R)= \begin{cases}
(R-r)^{d+\alpha}\quad &\text{if } \alpha \geq 1,\\
(R^\alpha-r^\alpha)^{(d+\alpha)/\alpha} &\text{if } \alpha <1.
\end{cases}$
\end{sa}

Let us recall that $\kappa$ would be different if $d\in \{1,2\}$. Next, we apply
a similar procedure for positive powers of supersolutions. As a
result we can estimate the $L^1$-norm of supersolutions $u$
from above by the $L^1$-norm of $u^p$ for small values of $p>0$ as in the
following theorem: 

\begin{sa}[Moser iteration for positive exponents]
\label{thm:p-est} Let $\frac12\leq r<R\leq 1$ and $p \in (0,\kappa^{-1})$ with
$\kappa=1+ \frac{\alpha}{d}$. Then there are constants $C,
\omega_1, \omega_2>0$ depending only on $d, \alpha_0, \Lambda$, such that for
every nonnegative supersolution $u$ in $Q= I\times \Omega$, $Q \Supset
Q_\oplus(R)$, the following estimate holds:
\begin{align}
\int_{Q_\oplus(r)} \widetilde u(t,x) \d x \d t &\leq \left(
\frac{C}{\bet{Q_\oplus(1)} G_2(r,R)}
\right)^{1/p-1} \left( \int_{Q_{\oplus}(R)} \widetilde u^{p}(t,x) \d x \d t
\right)^{1/p}, \label{eq:p-est}
\end{align}
where $\widetilde u = u + \norm{f}_{L^\infty(Q(R))}$ and  $ G_2(r,R) =
\begin{cases}
             (R-r)^{\omega_1}\quad &\text{if } \alpha \geq 1,\\
            (R-r)^{\omega_2} &\text{if } \alpha < 1.
            \end{cases}$
\end{sa}

At this stage we make an important observation. Since we are working with
supersolutions (and not with solutions) one cannot expect to estimate the
supremum of $\widetilde{u}$ in $Q_\oplus(r)$ but only the $L^1$-norm. Note that
for nonlocal equations, under our general assumptions, this is not possible
even for solutions. If it were correct, then we would be able to prove a Harnack
inequality it its strong formulation which we know to fail, see the discussion
in the introduction.

\subsection{Estimates of \texorpdfstring{$\log(u)$}{log(u)} for supersolutions
\texorpdfstring{$u$}{u}}

The estimates on $\log(u)$ are always quite delicate in the theory of De
Giorgi-Nash-Moser. In the case of second order differential operators,
the following inequality for nice positive functions $\psi, w$ is
important: 
\begin{align}
\int \nabla w \, \nabla (-\psi^2 w^{-1}) \geq \frac12 \int \psi^2 \left(\nabla
\log(w) \right)^2 - 2 \int (\nabla \psi)^2
\end{align}
In the nonlocal case we have a substitute for this inequality of the following
type:
\begin{align*}
\cE_t(w, -\psi^2 w^{-1}) \geq \iint \psi(x) \psi(y) \left(
\log \frac{w(t,y)}{\psi(y)} - \log \frac{w(t,x)}{\psi(x)} \right)^2 a(t,x,y)
\mu(x, \d y) dx - 3\ \cE_t(\psi,\psi)
\end{align*}

Using this observation and the technique of \cite{Mos71} we obtain the
following estimate of the level sets for $\log(u)$:

\begin{prop} \label{prop:loglemma} There is
$C=C(d,\alpha_0,\Lambda)>0$ such that for every supersolution $u$ of
\eqref{eq:par_equation} in $Q = (-1,1)\times B_2(0)$ which satisfies $u \geq
\epsilon >0$ in $(-1,1) \times \R^d$, there is a constant $a=a(\widetilde u) \in
\R$ such that the following inequalities hold
simultaneously:\enlargethispage{4em}
\begin{subequations}
\begin{align} 
\forall s >0 \colon (\d t \otimes \d x)\left( Q_{\oplus}(1) \cap \left\{ \log
\widetilde u < -s-a \right\} \right) \leq \frac{C \bet{B_1} }{s},
\label{eq:log-neg} \\
\forall s >0 \colon (\d t \otimes \d x)\left( Q_{\ominus}(1) \cap \left\{ \log
\widetilde u > s-a \right\} \right) \leq \frac{C \bet{B_1}}{s},
\label{eq:log-pos}
\end{align}
\end{subequations}
where $ \widetilde u = u + \norm{f}_{L^\infty(Q)}$.
\end{prop}

\subsection{The proof of \autoref{thm:harnack} and \autoref{thm:hoelder_main}}

Let us explain how the results of the previous subsection lead to a proof of
\autoref{thm:harnack}. 

\begin{proof}[Proof of \autoref{thm:harnack}]
The main idea is to apply the famous lemma of Bombieri-Giusti, proved in
\cite[pp. 731-733]{Mos71}. The version we use can be found in \cite[Section
2.2.3]{Sal02}.

Let $u$ as in the assumption and define $\widetilde u = u +
\norm{f}_{L^\infty(Q)}$. Without loss of generality we assume $\widetilde u
\geq \epsilon$ for some $\epsilon >0$. Next, we set $w = e^{-a}
\widetilde u^{-1}$ and $\widehat w = w^{-1} = e^{a} \widetilde u$,
where $a=a(\widetilde u)$ is chosen according to \autoref{prop:loglemma}, i.e.
there is $c_1>0$ such that for every $s>0$
\begin{align} \label{eq:horst}
\bet{ Q_{\oplus}(1) \cap \{\log w > s \} } \leq \frac{c_1 \bet{B_1}}{s},\quad
\text{and} \quad \bet{ Q_{\ominus}(1) \cap \{\log \widehat w > s \} } \leq
\frac{c_1 \bet{B_1}}{s}.
\end{align}
We apply the lemma of Bombier-Giusti twice: first to $w$ on a family of
domains  $\cU=(U(r))_{\frac{1}{2} \leq r \leq 1}$ and then to $\widehat w$ and a
family of domains ${\widehat \cU}=(\widehat U(r))_{\frac{1}{2} \leq r
\leq 1}$. \autoref{thm:p-est} and \autoref{thm:inf-est} are needed for
this step. Let us restrict to the
case $\alpha \geq 1$. In this case the
families $\cU$ and $\widehat \cU$ are given by
\begin{align*}
U(1) &= Q_{\oplus}(1),\quad & U(r) &= \left(1- {r^\alpha},
1\right) \times B_{r},\\
\widehat U(1) &= Q_{\ominus}(1), & \widehat
U(r)&=\left(-1, -1+r^\alpha \right) \times B_{r}
\end{align*}
We obtain
\begin{align*}
\sup_{U(\frac{1}{2})} w = e^{-a}  \sup_{U(\frac{1}{2})} \widetilde u^{-1} \leq C
\qquad \text{ and } \quad \norm{\widehat w}_{L^1(\widehat{U}(\frac{1}{2}))} =
e^{a} \norm{\widetilde u}_{L^1(\widehat{U}(\frac{1}{2}))} \leq \widehat C.
\end{align*}
Multiplying these two inequalities eliminates $a$ and yields
\[ \norm{\widetilde u}_{L^1(\widehat U(\frac{1}{2}))} \leq c_2
\inf_{U(\frac{1}{2})} \widetilde u\]
for a constant $c_2= C\, \widehat C$ that depends only on $d,\alpha_0$ and
$\Lambda$.
This proves \eqref{eq:harnack-speziell} in the case $\alpha \geq 1$. The case
$\alpha < 1$ is similar.
\end{proof}

We will not explain the details of how \autoref{thm:harnack} implies
\autoref{thm:hoelder_main}. This implication is straight-forward for
differential operators of second order. In the case of nonlocal operators one
has to take care of the fact that the auxiliary functions $M(t,x) = \sup_Q u -
u(t,x)$ and $m(t,x)=u - \inf_Q u$ are nonnegative in $Q$ but not in all of
$\R^d$. That is why one cannot apply \autoref{thm:harnack} directly for
$f=0$. The idea is to define $f(t,x) = (Lu^{-})(t,x)$ which is a bounded
function in a region if $u$ is nonnegative in a slightly larger region. This
step is carried out in \cite{Sil06} for elliptic equations and works well for
parabolic problems. 

Note that this approach requires 
\begin{align*}
\sup_{x \in B_2(0)} \int_{\R^d \setminus B_3(0)} \bet{x-y}^{\delta} \mu(x,\d y)
\leq C_0 \,. \tag{K$_3$} \label{eq:K-3}
\end{align*}
for some $\delta \in (0,1)$ and $C_0 \geq 1$. This condition follows from
\eqref{eq:K-1}. If one assumes \eqref{eq:K-1},
\eqref{eq:K-2} for bounded radii $\rho$ only, then one would need to assume
\eqref{eq:K-3} in order to
prove \autoref{thm:hoelder_main}.   

\bibliography{./lit_harnack}
\bibliographystyle{plain}

\end{document}